\documentclass[graybox]{svmult}

\usepackage{mathptmx}       
\usepackage{helvet}         
\usepackage{courier}        
\usepackage{type1cm}        
\usepackage{makeidx}         
\usepackage{graphicx}        
\usepackage{multicol}        
\usepackage[bottom]{footmisc}

\usepackage{algorithmic}
\usepackage{algorithm}

 \usepackage{amsmath}
 \usepackage{amssymb}
 \usepackage{mathrsfs}
\usepackage[pdf]{pstricks}
\usepackage{auto-pst-pdf}
\usepackage[hang]{subfigure}

\def\R		{\mathbb{R}}

\def\O		{\mathcal{O}}
\def\d		{\:\mathrm{d}}

\def\L		{\mathscr{L}}

\renewcommand{\vec}[1]{\mathbf{#1}}

\newcommand{\clifford}[1]{\ensuremath{C\kern-0.1em\ell_{#1}}}

\def\e		{\vec{e}}

\def\A		{\vec{A}}
\def\B		{\vec{B}}
\def\x		{\vec{x}}
\def\y		{\vec{y}}

\def\u		{\vec{u}}
\def\v		{\vec{v}}

 \newtheorem{thm}{Theorem}[section]
 \newtheorem{cor}[thm]{Corollary}
 \newtheorem{lem}[thm]{Lemma}
 \newtheorem{defn}[thm]{Definition}
 
 \newtheorem{ex}{Example}

\begin{document}

\title*{Detection of Outer Rotations on 3D-Vector Fields with Iterative Geometric Correlation}
\titlerunning{Detection of Outer Rotations with Geometric Correlation}
\author{Roxana Bujack, Gerik Scheuermann, and Eckhard Hitzer}
\authorrunning{Bujack, Scheuermann, Hitzer}
\institute{
Roxana Bujack \at Universit\"at Leipzig, Institut f\"ur Informatik, Johannisgasse 26, 04103 Leipzig, Germany\\ \email{bujack@informatik.uni-leipzig.de}
\and 
Gerik Scheuermann \at Universit\"at Leipzig, Institut f\"ur Informatik, Johannisgasse 26, 04103 Leipzig, Germany\\ \email{scheuermann@informatik.uni-leipzig.de}
\and
Eckhard Hitzer \at University of Fukui, Department of Applied Physics, 3-9-1 Bunkyo, Fukui 910-8507, Japan\\ \email{hitzer@mech.u-fukui.ac.jp}}
%
\maketitle

\abstract*{Correlation is a common technique for the detection of shifts. Its generalization to the multidimensional geometric correlation in Clifford algebras has proven a useful tool for color image processing, because it additionally contains information about rotational misalignment.
\newline\indent
In this paper we prove that applying the geometric correlation iteratively can detect the outer rotational misalignment for arbitrary three-dimensional vector fields. Thus, it develops a foundation applicable for image registration and pattern matching. Based on the theoretical work we have developed a new algorithm and tested it on some principle examples.}

\abstract{Correlation is a common technique for the detection of shifts. Its generalization to the multidimensional geometric correlation in Clifford algebras has proven a useful tool for color image processing, because it additionally contains information about rotational misalignment.
\newline\indent
In this paper we prove that applying the geometric correlation iteratively can detect the outer rotational misalignment for arbitrary three-dimensional vector fields. Thus, it develops a foundation applicable for image registration and pattern matching. Based on the theoretical work we have developed a new algorithm and tested it on some principle examples.}

\section{Introduction}
In signal processing correlation is a basic technique to determine the similarity or dissimilarity of two signals. It is widely used for image registration, pattern matching, and feature extraction \cite{BRO92,ZIT03}. The idea of using correlation for registration of shifted signals is that at the very position, where the signals match, the correlation function will 
take its maximum, compare \cite{RK82}.
\par 
For a long time the generalization of this method to multidimensional signals has only been an amount of single channel processes. The elements of 
Clifford algebras $\clifford{p,q}$, compare \cite{C1878,HS84}, have a natural geometric interpretation, so the analysis of multivariate signals expressed as multivector valued functions is a very reasonable approach.
\par
Scheuermann \cite{Sch99} used Clifford algebras for vector field analysis. Together with Ebling \cite{Ebl03,Ebl06} they developed a pattern matching algorithm based on geometric convolution and correlation and accelerated it by means of a Clifford Fourier transform and its convolution theorem.
\par
At about the same time Sangwine et. al. \cite{MSE01} introduced a generalized hypercomplex correlation for quaternions. Together with Ell and Moxey \cite{MSE02,MSE03} they used it to represent color images, interpreted as vector fields, geometrically. They discovered that this geometric correlation not only contains the translational difference of images given by the position of the magnitude peak, but also information about a possible rotational misalignment of two signals and showed how to apply them to approximately correct color space distortions. 
\par
Even though lately other approaches to work with color images were made \cite{GSH10,Schl11,MSM11} we want to extend the work and ideas of Moxey, Ell and Sangwine using hypercomplex correlation. We analyze vector fields $\v(\x):\R^m\to\R^3\subset\clifford{3,0}$ with values interpreted as elements of the geometric algebra $\clifford{3,0}$ and their copies produced from outer rotations. A great advantage of the geometric algebra is that many statements generally hold not just for vectors but for all multivectors. We will make use of that and state the more general formulae, whenever possible. 
\par
\begin{figure}[ht]
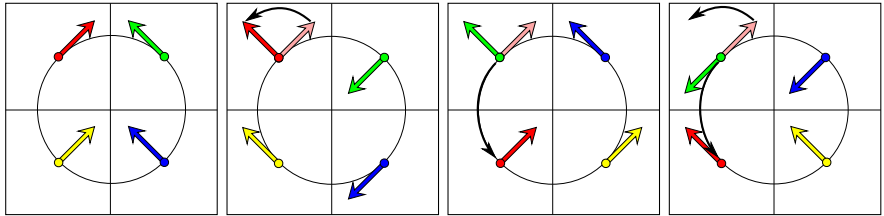

\begin{minipage}{0.23\textwidth}
\centering
\psset{unit=0.8pt}
\begin{pspicture}(102,102)
{
\newrgbcolor{curcolor}{1 1 1}
\pscustom[linestyle=none,fillstyle=solid,fillcolor=curcolor]
{
\newpath
\moveto(1,101)
\lineto(101,101)
\lineto(101,1)
\lineto(1,1)
\lineto(1,101)
\closepath
}
}
{
\newrgbcolor{curcolor}{0 0 0}
\pscustom[linewidth=0.5,linecolor=curcolor]
{
\newpath
\moveto(1,101)
\lineto(101,101)
\lineto(101,1)
\lineto(1,1)
\lineto(1,101)
\closepath
}
}
{
\newrgbcolor{curcolor}{1 1 1}
\pscustom[linestyle=none,fillstyle=solid,fillcolor=curcolor]
{
\newpath
\moveto(86.0166576,50.76017173)
\curveto(86.0166576,31.4302046)(70.3466235,15.76017013)(51.01665683,15.76017013)
\curveto(31.68669015,15.76017013)(16.01665605,31.4302046)(16.01665605,50.76017173)
\curveto(16.01665605,70.09013886)(31.68669015,85.76017333)(51.01665683,85.76017333)
\curveto(70.3466235,85.76017333)(86.0166576,70.09013886)(86.0166576,50.76017173)
\closepath
}
}
{
\newrgbcolor{curcolor}{0 0 0}
\pscustom[linewidth=0.09999999,linecolor=curcolor]
{
\newpath
\moveto(86.0166576,50.76017173)
\curveto(86.0166576,31.4302046)(70.3466235,15.76017013)(51.01665683,15.76017013)
\curveto(31.68669015,15.76017013)(16.01665605,31.4302046)(16.01665605,50.76017173)
\curveto(16.01665605,70.09013886)(31.68669015,85.76017333)(51.01665683,85.76017333)
\curveto(70.3466235,85.76017333)(86.0166576,70.09013886)(86.0166576,50.76017173)
\closepath
}
}
{
\newrgbcolor{curcolor}{1 0 0}
\pscustom[linestyle=none,fillstyle=solid,fillcolor=curcolor]
{
\newpath
\moveto(24.50853795,75.67407342)
\lineto(38.65067355,89.81620902)
\lineto(35.11513965,90.5233158)
\lineto(42.89331423,92.64463614)
\lineto(40.77199389,84.86646156)
\lineto(40.06488711,88.40199546)
\lineto(25.92275151,74.25985986)
\lineto(24.50853795,75.67407342)
\closepath
}
}
{
\newrgbcolor{curcolor}{0 0 0}
\pscustom[linewidth=0.1,linecolor=curcolor]
{
\newpath
\moveto(24.50853795,75.67407342)
\lineto(38.65067355,89.81620902)
\lineto(35.11513965,90.5233158)
\lineto(42.89331423,92.64463614)
\lineto(40.77199389,84.86646156)
\lineto(40.06488711,88.40199546)
\lineto(25.92275151,74.25985986)
\lineto(24.50853795,75.67407342)
\closepath
}
}
{
\newrgbcolor{curcolor}{1 0 0}
\pscustom[linestyle=none,fillstyle=solid,fillcolor=curcolor]
{
\newpath
\moveto(27.33696695,74.25986894)
\curveto(26.55591481,73.4788168)(25.28958199,73.47881391)(24.50853343,74.25986247)
\curveto(23.72748486,75.04091103)(23.72748776,76.30724385)(24.5085399,77.08829599)
\curveto(25.28959203,77.86934813)(26.55592485,77.86935102)(27.33697342,77.08830246)
\curveto(28.11802198,76.3072539)(28.11801908,75.04092108)(27.33696695,74.25986894)
\closepath
}
}
{
\newrgbcolor{curcolor}{0 0 0}
\pscustom[linewidth=0.1,linecolor=curcolor]
{
\newpath
\moveto(27.33696695,74.25986894)
\curveto(26.55591481,73.4788168)(25.28958199,73.47881391)(24.50853343,74.25986247)
\curveto(23.72748486,75.04091103)(23.72748776,76.30724385)(24.5085399,77.08829599)
\curveto(25.28959203,77.86934813)(26.55592485,77.86935102)(27.33697342,77.08830246)
\curveto(28.11802198,76.3072539)(28.11801908,75.04092108)(27.33696695,74.25986894)
\closepath
}
}
{
\newrgbcolor{curcolor}{0 0 0}
\pscustom[linewidth=0.1,linecolor=curcolor]
{
\newpath
\moveto(1,50.5)
\lineto(101,50.5)
}
}
{
\newrgbcolor{curcolor}{0 0 0}
\pscustom[linewidth=0.2,linecolor=curcolor]
{
\newpath
\moveto(50.5,101)
\lineto(50.5,1)
}
}
{
\newrgbcolor{curcolor}{0 1 0}
\pscustom[linestyle=none,fillstyle=solid,fillcolor=curcolor]
{
\newpath
\moveto(76.0166592,74.34596057)
\lineto(61.8745236,88.48809617)
\lineto(61.16741682,84.95256227)
\lineto(59.04609648,92.73073685)
\lineto(66.82427106,90.60941651)
\lineto(63.28873716,89.90230973)
\lineto(77.43087276,75.76017413)
\lineto(76.0166592,74.34596057)
\closepath
}
}
{
\newrgbcolor{curcolor}{0 0 0}
\pscustom[linewidth=0.1,linecolor=curcolor]
{
\newpath
\moveto(76.0166592,74.34596057)
\lineto(61.8745236,88.48809617)
\lineto(61.16741682,84.95256227)
\lineto(59.04609648,92.73073685)
\lineto(66.82427106,90.60941651)
\lineto(63.28873716,89.90230973)
\lineto(77.43087276,75.76017413)
\lineto(76.0166592,74.34596057)
\closepath
}
}
{
\newrgbcolor{curcolor}{0 1 0}
\pscustom[linestyle=none,fillstyle=solid,fillcolor=curcolor]
{
\newpath
\moveto(77.43086368,77.17438956)
\curveto(78.21191581,76.39333743)(78.21191871,75.12700461)(77.43087015,74.34595604)
\curveto(76.64982158,73.56490748)(75.38348876,73.56491038)(74.60243663,74.34596251)
\curveto(73.82138449,75.12701465)(73.8213816,76.39334747)(74.60243016,77.17439603)
\curveto(75.38347872,77.9554446)(76.64981154,77.9554417)(77.43086368,77.17438956)
\closepath
}
}
{
\newrgbcolor{curcolor}{0 0 0}
\pscustom[linewidth=0.1,linecolor=curcolor]
{
\newpath
\moveto(77.43086368,77.17438956)
\curveto(78.21191581,76.39333743)(78.21191871,75.12700461)(77.43087015,74.34595604)
\curveto(76.64982158,73.56490748)(75.38348876,73.56491038)(74.60243663,74.34596251)
\curveto(73.82138449,75.12701465)(73.8213816,76.39334747)(74.60243016,77.17439603)
\curveto(75.38347872,77.9554446)(76.64981154,77.9554417)(77.43086368,77.17438956)
\closepath
}
}
{
\newrgbcolor{curcolor}{1 1 0}
\pscustom[linestyle=none,fillstyle=solid,fillcolor=curcolor]
{
\newpath
\moveto(24.60244795,25.63517342)
\lineto(38.74458355,39.77730902)
\lineto(35.20904965,40.4844158)
\lineto(42.98722423,42.60573614)
\lineto(40.86590389,34.82756156)
\lineto(40.15879711,38.36309546)
\lineto(26.01666151,24.22095986)
\lineto(24.60244795,25.63517342)
\closepath
}
}
{
\newrgbcolor{curcolor}{0 0 0}
\pscustom[linewidth=0.1,linecolor=curcolor]
{
\newpath
\moveto(24.60244795,25.63517342)
\lineto(38.74458355,39.77730902)
\lineto(35.20904965,40.4844158)
\lineto(42.98722423,42.60573614)
\lineto(40.86590389,34.82756156)
\lineto(40.15879711,38.36309546)
\lineto(26.01666151,24.22095986)
\lineto(24.60244795,25.63517342)
\closepath
}
}
{
\newrgbcolor{curcolor}{1 1 0}
\pscustom[linestyle=none,fillstyle=solid,fillcolor=curcolor]
{
\newpath
\moveto(27.43087695,24.22096894)
\curveto(26.64982481,23.4399168)(25.38349199,23.43991391)(24.60244343,24.22096247)
\curveto(23.82139486,25.00201103)(23.82139776,26.26834385)(24.6024499,27.04939599)
\curveto(25.38350203,27.83044813)(26.64983485,27.83045102)(27.43088342,27.04940246)
\curveto(28.21193198,26.2683539)(28.21192908,25.00202108)(27.43087695,24.22096894)
\closepath
}
}
{
\newrgbcolor{curcolor}{0 0 0}
\pscustom[linewidth=0.1,linecolor=curcolor]
{
\newpath
\moveto(27.43087695,24.22096894)
\curveto(26.64982481,23.4399168)(25.38349199,23.43991391)(24.60244343,24.22096247)
\curveto(23.82139486,25.00201103)(23.82139776,26.26834385)(24.6024499,27.04939599)
\curveto(25.38350203,27.83044813)(26.64983485,27.83045102)(27.43088342,27.04940246)
\curveto(28.21193198,26.2683539)(28.21192908,25.00202108)(27.43087695,24.22096894)
\closepath
}
}
{
\newrgbcolor{curcolor}{0 0 1}
\pscustom[linestyle=none,fillstyle=solid,fillcolor=curcolor]
{
\newpath
\moveto(76.0166592,24.34596057)
\lineto(61.8745236,38.48809617)
\lineto(61.16741682,34.95256227)
\lineto(59.04609648,42.73073685)
\lineto(66.82427106,40.60941651)
\lineto(63.28873716,39.90230973)
\lineto(77.43087276,25.76017413)
\lineto(76.0166592,24.34596057)
\closepath
}
}
{
\newrgbcolor{curcolor}{0 0 0}
\pscustom[linewidth=0.1,linecolor=curcolor]
{
\newpath
\moveto(76.0166592,24.34596057)
\lineto(61.8745236,38.48809617)
\lineto(61.16741682,34.95256227)
\lineto(59.04609648,42.73073685)
\lineto(66.82427106,40.60941651)
\lineto(63.28873716,39.90230973)
\lineto(77.43087276,25.76017413)
\lineto(76.0166592,24.34596057)
\closepath
}
}
{
\newrgbcolor{curcolor}{0 0 1}
\pscustom[linestyle=none,fillstyle=solid,fillcolor=curcolor]
{
\newpath
\moveto(77.43086368,27.17438956)
\curveto(78.21191581,26.39333743)(78.21191871,25.12700461)(77.43087015,24.34595604)
\curveto(76.64982158,23.56490748)(75.38348876,23.56491038)(74.60243663,24.34596251)
\curveto(73.82138449,25.12701465)(73.8213816,26.39334747)(74.60243016,27.17439603)
\curveto(75.38347872,27.9554446)(76.64981154,27.9554417)(77.43086368,27.17438956)
\closepath
}
}
{
\newrgbcolor{curcolor}{0 0 0}
\pscustom[linewidth=0.1,linecolor=curcolor]
{
\newpath
\moveto(77.43086368,27.17438956)
\curveto(78.21191581,26.39333743)(78.21191871,25.12700461)(77.43087015,24.34595604)
\curveto(76.64982158,23.56490748)(75.38348876,23.56491038)(74.60243663,24.34596251)
\curveto(73.82138449,25.12701465)(73.8213816,26.39334747)(74.60243016,27.17439603)
\curveto(75.38347872,27.9554446)(76.64981154,27.9554417)(77.43086368,27.17438956)
\closepath
}
}
\end{pspicture}
\end{minipage}
\hspace{0.1cm}
\begin{minipage}{0.23\textwidth}
\centering
\psset{unit=0.8pt}
  \input{vf_outer}
\end{minipage}
\hspace{0.1cm}
\begin{minipage}{0.23\textwidth}
\centering
\psset{unit=0.8pt}
  \input{vf_inner}
\end{minipage}
\hspace{0.1cm}
\begin{minipage}{0.23\textwidth}
\centering
\psset{unit=0.8pt}
  \input{vf_total}
\end{minipage}
\caption{From left to right: a vector field, its copy from outer rotation, inner rotation, total rotation.}\label{f:1}
\end{figure}
The term rotational misalignment with respect to multivector fields is ambiguous. In general three types of rotations are distinguished, compare Figure \ref{f:1}. Let $\operatorname{R} _{\vec P,\alpha}$ be an operator, that describes a mathematically positive rotation by the angle $\alpha\in[0,\pi]$\footnote{As in \cite{MSE02} we encode the sign in the bivector $\vec P$ and deal with positive angles only.} in the plane $P$, spanned by the unit bivector $\vec P$. We say two multivector fields $\A(\x),\B(\x):\R^m\to\clifford{3,0}$ differ by an \textbf{outer rotation} if they suffice
\begin{equation}
\begin{aligned}\label{outer}
 \A(\x)=\operatorname{R} _{\vec P,\alpha}(\B(\x)).
\end{aligned}
\end{equation}
Independent from their position $\x$ the multivector $\A(\x)$ is the rotated copy of the multivector $\B(\x)$. This kind of rotation appears for example in color images, when the color vector space is turned, but the picture is not moved, compare \cite{MSE03}. 
In contrast to that for $m\leq3$ an \textbf{inner rotation} is described by 
\begin{equation}
\begin{aligned}\A(\x)=\B(\operatorname{R} _{\vec P,-\alpha}(\x)).
\end{aligned}
\end{equation}
Here the starting position of every vector is rotated by $-\alpha$ then the old vector is reattached at the new position. It still points into the old direction. The inner rotation is suitable to describe the rotation of a color image. The color does not change when the picture is turned. In the case of bijective fields $\A(\x),\B(\x):\R^3\to\R^3\subset\clifford{3,0}$ a \textbf{total rotation} is a combination of the previous ones defined by 
\begin{equation}
\begin{aligned}\A(\x)=\operatorname{R} _{\vec P,\alpha}(\B(\operatorname{R} _{\vec P,-\alpha}(\x))).
\end{aligned}
\end{equation} 
It can be interpreted as coordinate transform, that means as looking at the multivector field from another point of view. The positions and the multivectors are stiffly connected during the rotation.
\par
With respect to the definition of the correlation there are different formulae in current literature, \cite{Ebl06,MSE03}. We prefer the following one because it satisfies a geometric generalization of the Wiener-Khinchin theorem and because it coincides with the definition of the standard cross-correlation for complex functions in the special case of $\clifford{0,1}$, \cite{RK82}. For vector fields they mostly coincide anyways because of $\overline{\v(\x)}=\v(\x)$, where the overbar denotes reversion.
\begin{defn}
 The \textbf{geometric cross correlation} of two multivector valued functions $\A(\x),\B(\x):\R^m\to\clifford{p,q}$ is a multivector valued function defined by
\begin{equation}
\begin{aligned}
(\A\star \B)(\x):=&\int_{\R^m}\overline{\A(\y)}\B(\y+\x)\d^m\y,
\end{aligned}
\end{equation} 
where $\overline{\A(\y)}$ denotes the reversion $\sum\limits_{k=0}^n(-1)^{\frac12k(k-1)}\langle \A(\y)\rangle_k$.
\end{defn}
To simplify the notation we will only analyze the correlation at the origin. If the vector fields also differ by an inner shift this can for example be detected by evaluating the magnitude of the correlation \cite{Hes86} or phase correlation of the field magnitudes \cite{KH75}. 
Our methods can then be applied analogously to that translated position.
\section{Motivation}
In two dimensions a mathematically positive\footnote{anticlockwise} outer rotation of a vector field $\R^2\to\R^2\subset\clifford{2,0}$ by the angle $\alpha$ takes the shape
\begin{equation}
\begin{aligned}
\operatorname{R} _{\e_{12},\alpha}(\v(\x))=e^{-\alpha \e_{12}}\v(\x).
\end{aligned} 
\end{equation}
So the product of the vector field and its copy at any position $\x\in\R^m$ yields
\begin{equation}
\begin{aligned}\label{outerprod}
\operatorname{R} _{\e_{12},\alpha}(\v(\x))\v(\x)=e^{-\alpha \e_{12}}\v(\x)\v(\x)=\v(\x)^2e^{-\alpha \e_{12}},
\end{aligned} 
\end{equation}
with $\v(\x)^2=\v(\x)\overline{\v(\x)}=||\v(\x)||_2^2\in\R$ and the rotation can fully be restored by rotating back with the inverse of (\ref{outerprod}) or explicitly calculating $\alpha$ as described in \cite{Hes86}. This property is inherited by the geometric correlation at the origin
\begin{equation}
\begin{aligned}
(\operatorname{R} _{\e_{12},\alpha}(\v)\star \v)(0)&=\int_{\R^m}\overline{\operatorname{R} _{\e_{12},\alpha}(\v(\x))}\v(\x)\d^m\x
=||\v(\x)||_{L^2}^2e^{-\alpha \e_{12}},
\end{aligned} 
\end{equation}
which is to be preferred because of its robustness.
\par
In three dimensions not only the angle but also the plane of rotation $P$ has to be detected in order to reconstruct the whole transform. We want to analyze if the geometric correlation at the origin contains enough information here, too. First we look at two vectors $\u,\v$ that suffice $\u=\operatorname{R} _{\vec P,\alpha}(\v)$. Their geometric product 
\begin{equation}
\begin{aligned}\label{vectors}
\u\v=&\u\cdot \v+\u\wedge \v
=|\u||\v|\big(\cos(\angle(\u,\v))+\sin(\angle(\u,\v))\frac{\u\wedge \v}{|\u\wedge \v|}\big)
\end{aligned}
\end{equation}
contains an angle and a plane and therefore seems very motivating. But the rotation $\operatorname{R} _{\vec P,\alpha}$ we used is not necessarily the one which is described by $\operatorname{R}_{\frac{\u\wedge \v}{|\u\wedge \v|},\angle(\u,\v)}$. These two rotations only coincide if the vectors lie completely within the plane $P$. The reason for that is as follows. A vector and its rotated copy do not contain enough information to reconstruct the rotation that produced the copy. Figure \ref{f:2} shows some of the infinitely many different rotations that can result in the same copy. Regard the set of all circles $C$, that contain the end points of $\u$ and $\v$ and are located on the sphere $S_{|\u|}(0)$ with radius $r=|\u|=|\v|$ centered at the origin. Every plane that includes a circle in $C$ is a possible plane for the rotation from $\v$ to $\u$.
The information we get out of the geometric product belongs to the plane that fully contains both vectors. This rotation is the one that has the smallest angle of all possible ones and forms the largest circle on a great circle of the sphere, shown on the very left in Figure \ref{f:2}.
\par
\begin{figure}[ht]
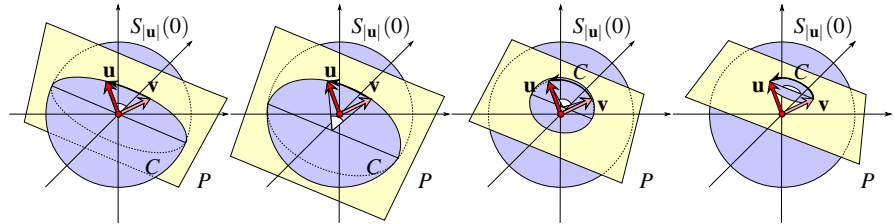

\begin{minipage}{0.23\textwidth}
\centering
\psset{unit=0.55pt}
  \input{Rdrehung1}
\end{minipage}
\hspace{0.1cm}
\begin{minipage}{0.23\textwidth}
\centering
\psset{unit=0.55pt}
  \input{Rdrehung2}
\end{minipage}
\hspace{0.1cm}
\begin{minipage}{0.23\textwidth}
\centering
\psset{unit=0.55pt}
  \input{Rdrehung3}
\end{minipage}
\hspace{0.1cm}
\begin{minipage}{0.23\textwidth}
\centering
\psset{unit=0.55pt}
  \input{Rdrehung4}
\end{minipage}
\caption{Different rotations of a vector $\u$ that lead to the same result $\v$.}\label{f:2}
\end{figure}
For the product of just two vectors the information from the geometric product is sufficient to realign them, but for a whole vector field the detected rotation from the correlation will in general not be the correct one. Moxey et. al. already stated in \cite{MSE03} that the hypercomplex correlation can effectively compute the rotation over two images, but that the perfect mapping can only be found, if specific conditions hold, for example if the images consist of one color only.
%

\begin{figure}[ht]
\centering
\subfigure[Vector field (\ref{bsp1a})]{\includegraphics[width=0.47\textwidth]{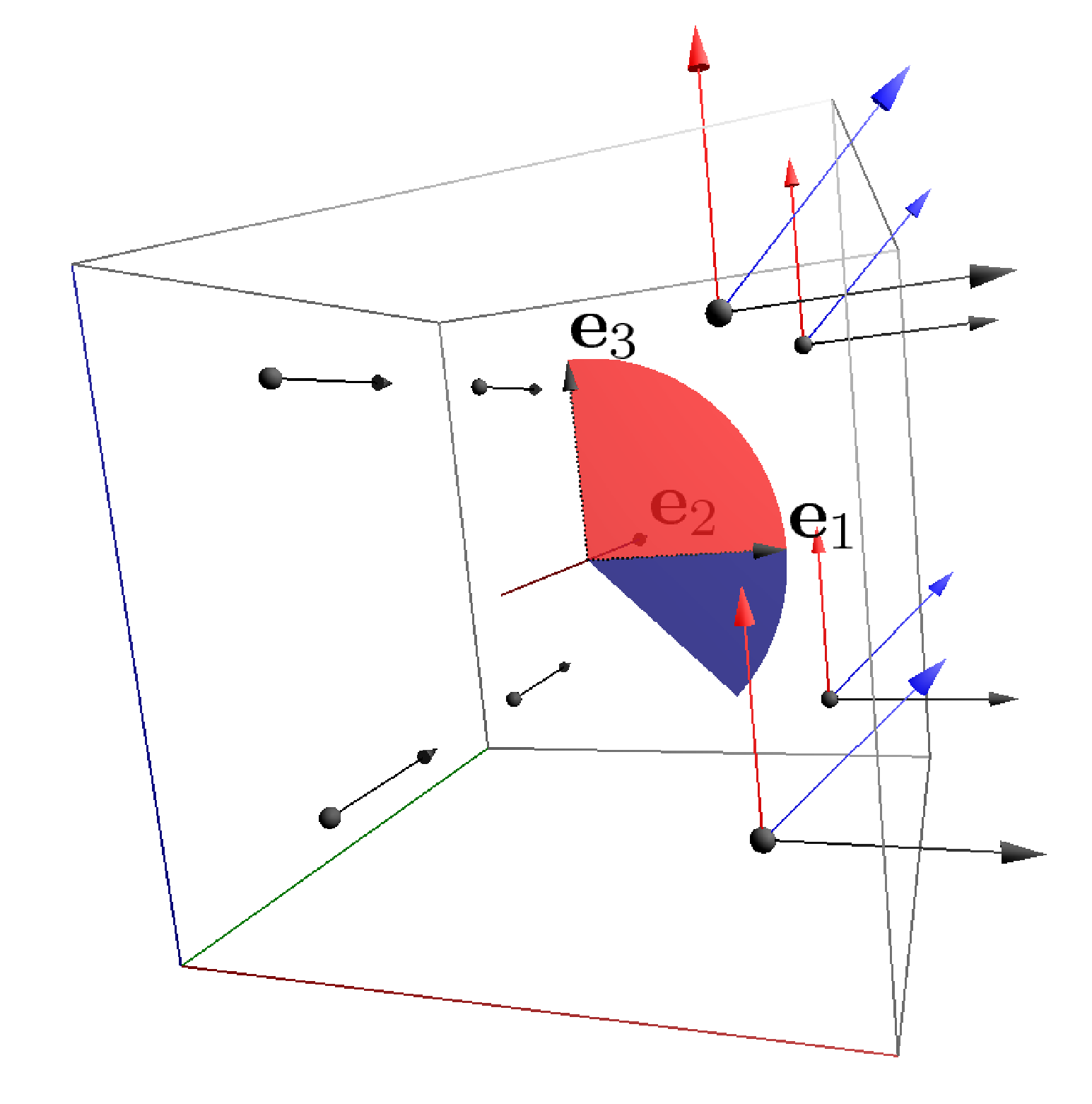}}\quad
\subfigure[Vector field (\ref{bsp2a})]{\includegraphics[width=0.47\textwidth]{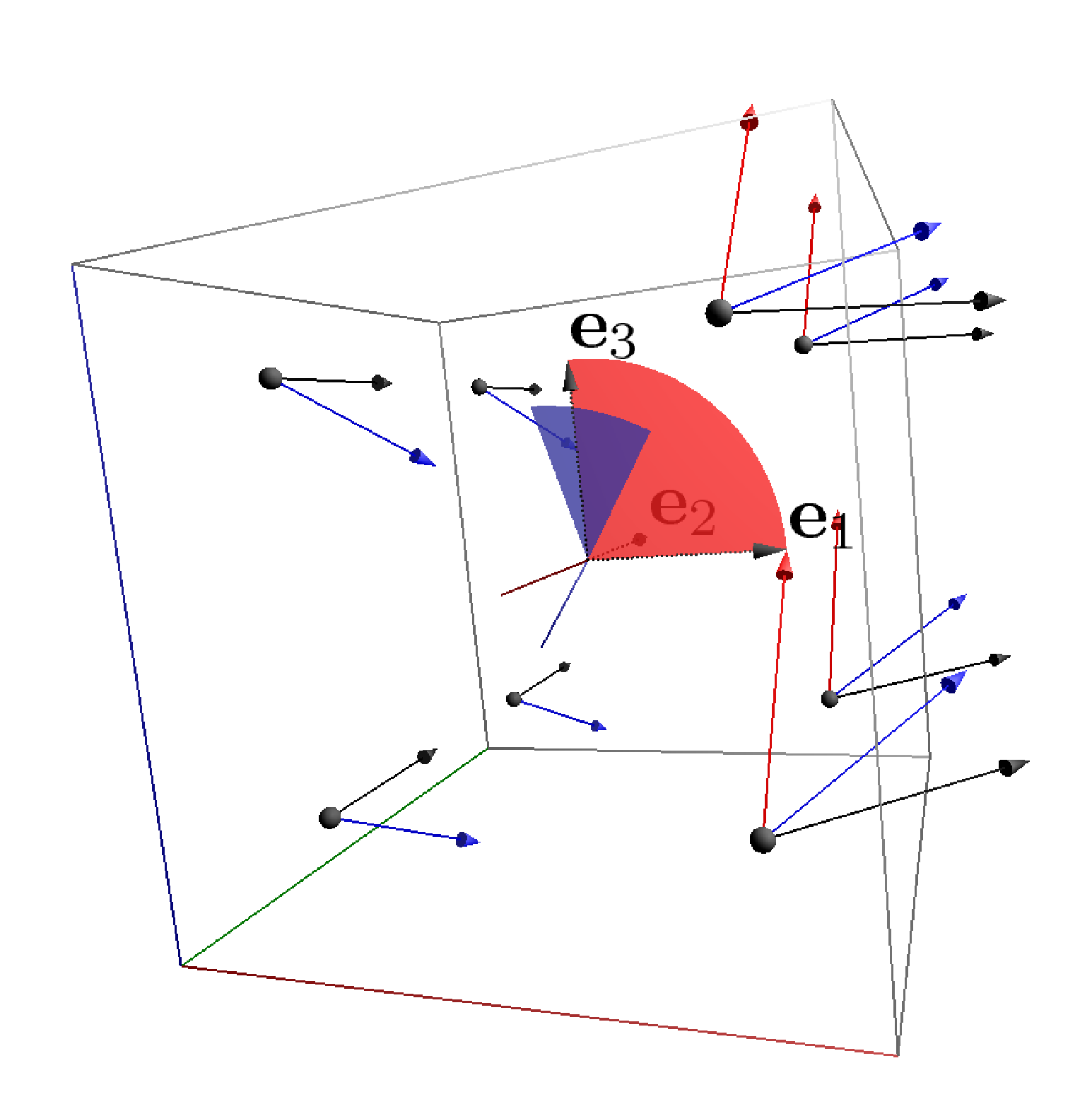}}
\caption{Example result of the geometric correlation. Original fields are depicted in black, rotated copies in red, corrected fields after application of the correlation rotor in blue using CLUCalc \cite{Per09}.}\label{f:3}
\end{figure}
\begin{ex}
The geometric correlation of the vector field
\begin{equation}
\begin{aligned}\label{bsp1a}
\v(\x)=&\begin{cases}
         \e_1,&\text{ for }x_1,x_2,x_3\in(-1,1), x_1\geq 0,\\
\e_2,&\text{ for }x_1,x_2,x_3\in(-1,1), x_1<0,\\
0,&\text{ else,}\\
        \end{cases}
\end{aligned}
\end{equation}
and its copy rotated by $\operatorname{R}_{\e_{13},\frac\pi2}$, which are shown on the left of Figure \ref{f:3}, suffices
\begin{equation}
\begin{aligned}
(\operatorname{R}_{\e_{13},\frac\pi2}(\v)\star\v)(0)
=&\int_{-1}^1\int_{-1}^1( \int_{0}^1\e_3\e_1 \d x_1
+\int_{-1}^0 \e_2\e_2\d x_1)\d x_3\d x_2\\
=&-4\e_{13}+4
\\=&\sqrt{32}e^{-\frac\pi4\e_{13}}.
\end{aligned}
\end{equation}
Here the unit bivector $\e_{13}$ indeed describes the rotational plane we looked for, but the angle $\frac\pi4$ is only half the angle of the original rotation. The result is that the restored field and the original one do not match.
\par
In general the correlation even detects a wrong rotational plane, consider for example
\begin{equation}
\begin{aligned}
\v(\x)=&\begin{cases}\label{bsp2a}
         \e_1+\e_2,&\text{ for }x_1,x_2,x_3\in(-1,1), x_1\geq 0,\\
\e_2,&\text{ for }x_1,x_2,x_3\in(-1,1), x_1<0,\\
0,&\text{ else,}\\
        \end{cases}
\end{aligned}
\end{equation}
rotated by $\operatorname{R}_{\e_{13},\frac\pi2}$, depicted on the right of Figure \ref{f:3}. The geometric correlation will return the plane spanned by $\e_{12}+\e_{13}+\e_{23}$ and the angle $\arctan(\frac{\sqrt{3}}{2})$, which are both incorrect.
\end{ex}
In the following section we will analyze how effective the correlation can calculate the rotation and prove, that despite the impression, the previous example gives, the geometric correlation contains enough information to reconstruct the misalignment.
\section{Outer Rotation and the Geometric Correlation}
The three-dimensional mathematically positive outer rotation (\ref{outer}) of a vector field by the angle $\alpha\in[0,\pi]$ along the plane $P$, spanned by the unit bivector $\vec P$, takes the shape
\begin{equation}
\begin{aligned}
\operatorname{R} _{P,\alpha}(\v(\x))=&e^{-\frac\alpha2 \vec P}\v(\x)e^{\frac\alpha2 \vec P}
=e^{-\alpha \vec P}\v_{\parallel\vec P}(\x)+\v_{\perp\vec P}(\x).
\end{aligned} 
\end{equation}
So the product of the vector field and its copy at any position
\begin{equation}
\begin{aligned}\label{outer_prod}
\operatorname{R} _{P,\alpha}(\v(\x))\v(\x)=&
e^{-\alpha \vec P}\v_{\parallel\vec P}(\x)^2+(e^{-\alpha \vec P}-1)\v_{\parallel\vec P}(\x)\v_{\perp\vec P}(\x)+\v_{\perp\vec P}(\x)^2
\end{aligned} 
\end{equation}
does usually not simply yield the rotation we looked for, like in the two-dimensional case, but a rather good approximation depending on the parallel and the orthogonal parts of the vector field with respect to the plane of rotation. In order to keep the notation short we partly drop the argument $\x$ of the vector fields $\v(\x)$ by just writing $\v$ and assume without loss of generality
\begin{equation}
\begin{aligned}\label{assumption}
||\operatorname{R} _{P,\alpha}(\v)||_{L^2}^2=||\v||_{L^2}^2=||\v_{\parallel\vec P}||_{L^2}^2+||\v_{\perp\vec P}||_{L^2}^2=1.
\end{aligned} 
\end{equation}
\begin{lem}\label{l:outer_prod}
Let $\v\in L^2(\R^m,\R^{3,0}\subset\clifford{3,0})$ be a square integrable vector field and $\operatorname{R} _{P,\alpha}(\v)$ its copy from an outer rotation. The rotational misalignment of $\operatorname{R} _{P,\alpha}(\v)$ does not increase if we apply the outer rotation encoded in the normalized geometric cross correlation 
\begin{equation}
\begin{aligned}
\frac{(\operatorname{R} _{P,\alpha}(\v)\star\v)(0)}{{||\operatorname{R} _{P,\alpha}(\v)||_{L^2}||\v||_{L^2}}}.
\end{aligned} 
\end{equation}
\end{lem}
\begin{proof}
We denote the polar form of the normalized geometric cross correlation by $e^{\varphi\vec Q}$ with the unit bivector $\vec Q$ and $\varphi\in[0,\pi]$. So using (\ref{outer_prod}) and the assumption (\ref{assumption}) we get
\begin{equation}
\begin{aligned}
e^{\varphi\vec Q}
=&(\operatorname{R} _{P,\alpha}(\v)\star\v)(0)
\\=&
\int e^{-\alpha \vec P}\v_{\parallel\vec P}(\x)^2+(e^{-\alpha \vec P}-1)\v_{\parallel\vec P}(\x)\v_{\perp\vec P}(\x)
+\v_{\perp\vec P}(\x)^2\d^m \x
\\=&e^{-\alpha \vec P}||\v_{\parallel\vec P}||_{L^2}^2+(e^{-\alpha \vec P}-1)\int\v_{\parallel\vec P}\v_{\perp\vec P}\d \x
+||\v_{\perp\vec P}||_{L^2}^2
\end{aligned} 
\end{equation}
with the scalar part
\begin{equation}
\begin{aligned}\label{scphiq}
\cos(\varphi)=&\langle e^{\varphi\vec Q}\rangle_0
=\cos(\alpha)||\v_{\parallel\vec P}||_{L^2}^2+||\v_{\perp\vec P}||_{L^2}^2
\end{aligned} 
\end{equation}
and the bivector part
\begin{equation}
\begin{aligned}\label{bivphiq}
\sin(\varphi)\vec Q&=\langle e^{\varphi\vec Q}\rangle_2
\\&=-\sin(\alpha) \vec P||\v_{\parallel\vec P}||_{L^2}^2+(-\sin(\alpha) \vec P+\cos(\alpha)-1)\int\v_{\parallel\vec P}\v_{\perp\vec P}\d \x
\end{aligned} 
\end{equation}
with squared magnitude
\begin{equation}
\begin{aligned}
||\langle e^{\varphi\vec Q}\rangle_2||^2
=&\sin(\alpha)^2||\v_{\parallel\vec P}||_{L^2}^4+(2-2\cos(\alpha))\,||\int\v_{\parallel\vec P}\v_{\perp\vec P}\d \x||^2.
\end{aligned} 
\end{equation}
Thats why we know the explicit expressions for
\begin{equation}
\begin{aligned}\label{bivphiq2}
\vec Q
=&\frac{\langle e^{\varphi\vec Q}\rangle_2}{||\langle e^{\varphi\vec Q}\rangle_2||}
\end{aligned} 
\end{equation}
and
\begin{equation}
\begin{aligned}
\varphi
=&\operatorname{atan2}(||\langle e^{\varphi\vec Q}\rangle_2||,\langle e^{\varphi\vec Q}\rangle_0).
\end{aligned} 
\end{equation}
The outer rotation encoded in the correlation applied to $\operatorname{R} _{P,\alpha}(\v)$ takes the shape
\begin{equation}
\begin{aligned}
e^{-\frac\varphi2\vec Q}\operatorname{R} _{P,\alpha}(\v)e^{\frac\varphi2\vec Q}=&e^{-\frac\varphi2\vec Q}e^{-\frac\alpha2\vec P}\v e^{\frac\alpha2\vec P}e^{\frac\varphi2\vec Q}.
\end{aligned} 
\end{equation}
The composition of the two rotations is a rotation itself. It shall be written as $e^{-\frac\beta2\vec R}\v e^{\frac\beta2\vec R}$ with the unit bivector $\vec R$ and angle $\beta\in[0,\pi]$, so we get the relation
\begin{equation}
\begin{aligned}\label{beta_def}
e^{\frac\beta2\vec R}=e^{\frac\alpha2\vec P}e^{\frac\varphi2\vec Q}.
\end{aligned} 
\end{equation}
To determine whether or not it is smaller than the original one, it is sufficient to compare $\beta$ and $\alpha$. We will prove that $\beta\leq\alpha$ by proving the inequality
\begin{equation}
\begin{aligned}\label{beta<alpha}
\frac\beta2=\arg(e^{\frac\beta2\vec R})=\arg(e^{\frac\alpha2\vec P}e^{\frac\varphi2\vec Q})\leq\arg(e^{\frac\alpha2\vec P})=\frac\alpha2.
\end{aligned} 
\end{equation}
We evaluate (\ref{beta_def}) by inserting   
(\ref{bivphiq2}) and (\ref{bivphiq}) 
and get
\begin{equation}
\begin{aligned}\label{beta3}
e^{\frac\beta2\vec R}
=&e^{\frac\alpha2\vec P}e^{\frac\varphi2\vec Q}
\\=&\cos(\frac\alpha2)\cos(\frac\varphi2)+\cos(\frac\alpha2)\sin(\frac\varphi2)\vec Q+\sin(\frac\alpha2)\cos(\frac\varphi2)\vec P+\sin(\frac\alpha2)\sin(\frac\varphi2)\vec P\vec Q\\
=&\cos(\frac\alpha2)\cos(\frac\varphi2)+\sin(\frac\alpha2)\cos(\frac\varphi2)\vec P
+\frac{\sin(\frac\varphi2)}{||\langle e^{\varphi\vec Q}\rangle_2||}\big(-\cos(\frac\alpha2)\sin(\alpha) \vec P||\v_{\parallel\vec P}||_{L^2}^2
\\&+\cos(\frac\alpha2)(\cos(\alpha)-1)\int\v_{\parallel\vec P}\v_{\perp\vec P}\d \x
-\cos(\frac\alpha2)\sin(\alpha) \vec P\int\v_{\parallel\vec P}\v_{\perp\vec P}\d \x
\\&+\sin(\frac\alpha2)\sin(\alpha)||\v_{\parallel\vec P}||_{L^2}^2
+\sin(\frac\alpha2)(\cos(\alpha)-1)\vec P\int\v_{\parallel\vec P}\v_{\perp\vec P}\d \x
\\&+\sin(\frac\alpha2)\sin(\alpha) \int\v_{\parallel\vec P}\v_{\perp\vec P}\d \x\big),
\end{aligned} 
\end{equation}
and applying addition theorems 
on the $\int\v_{\parallel\vec P}\v_{\perp\vec P}\d \x$-parts 
leads to
\begin{equation}
\begin{aligned}\label{beta4}
&\big(\cos(\frac\alpha2)(\cos(\alpha)-1)\int\v_{\parallel\vec P}\v_{\perp\vec P}\d \x+\sin(\frac\alpha2)\sin(\alpha)\int\v_{\parallel\vec P}\v_{\perp\vec P}\d \x \big)\frac{\sin(\frac\varphi2)}{||\langle e^{\varphi\vec Q}\rangle_2||}
\\&=\big(\cos(\frac\alpha2)\cos(\alpha)+\sin(\frac\alpha2)\sin(\alpha)-\cos(\frac\alpha2)\big)\frac{\sin(\frac\varphi2)\int\v_{\parallel\vec P}\v_{\perp\vec P}\d \x}{||\langle e^{\varphi\vec Q}\rangle_2||}
\\&=\big(\cos(\frac\alpha2-\alpha)-\cos(\frac\alpha2)\big)\frac{\sin(\frac\varphi2)\int\v_{\parallel\vec P}\v_{\perp\vec P}\d \x}{||\langle e^{\varphi\vec Q}\rangle_2||}
\\&=0,
\end{aligned} 
\end{equation}
and on the $\vec P\int\v_{\parallel\vec P}\v_{\perp\vec P}\d \x$-parts to
\begin{equation}
\begin{aligned}\label{beta5}
&\big(\sin(\frac\alpha2)(\cos(\alpha)-1)\vec P\int\v_{\parallel\vec P}\v_{\perp\vec P}\d \x
-\cos(\frac\alpha2)\sin(\alpha) \vec P\int\v_{\parallel\vec P}\v_{\perp\vec P}\d \x\big)\frac{\sin(\frac\varphi2)}{||\langle e^{\varphi\vec Q}\rangle_2||}
\\&=\big(\sin(\frac\alpha2)\cos(\alpha)-\cos(\frac\alpha2)\sin(\alpha)-\sin(\frac\alpha2)\big)\frac{\sin(\frac\varphi2)\vec P\int\v_{\parallel\vec P}\v_{\perp\vec P}\d \x}{||\langle e^{\varphi\vec Q}\rangle_2||}
\\&=\big(\sin(\frac\alpha2-\alpha)-\sin(\frac\alpha2)\big)\frac{\sin(\frac\varphi2)\vec P\int\v_{\parallel\vec P}\v_{\perp\vec P}\d \x}{||\langle e^{\varphi\vec Q}\rangle_2||}
\\&=-2\sin(\frac\alpha2)\frac{\sin(\frac\varphi2)\vec P\int\v_{\parallel\vec P}\v_{\perp\vec P}\d \x}{||\langle e^{\varphi\vec Q}\rangle_2||}.
\end{aligned} 
\end{equation}
We insert (\ref{beta4}) and (\ref{beta5}) in (\ref{beta3}) and get
\begin{equation}
\begin{aligned}
e^{\frac\beta2\vec R}
=&\cos(\frac\alpha2)\cos(\frac\varphi2)+\sin(\frac\alpha2)\cos(\frac\varphi2)\vec P
\\&+\frac{\sin(\frac\varphi2)}{||\langle e^{\varphi\vec Q}\rangle_2||}\big(-\cos(\frac\alpha2)\sin(\alpha) \vec P||\v_{\parallel\vec P}||_{L^2}^2
\\&+\sin(\frac\alpha2)\sin(\alpha)||\v_{\parallel\vec P}||_{L^2}^2
-2\sin(\frac\alpha2)\vec P\int\v_{\parallel\vec P}\v_{\perp\vec P}\d \x\big).
\end{aligned} 
\end{equation}
Its scalar part
\begin{equation}
\begin{aligned}\label{sc}
\langle e^{\frac\beta2\vec R}\rangle_0
=&\cos(\frac\alpha2)\cos(\frac\varphi2)
+\frac{1}{||\langle e^{\varphi\vec Q}\rangle_2||}\sin(\frac\alpha2)\sin(\frac\varphi2)\sin(\alpha)||\v_{\parallel\vec P}||_{L^2}^2
\end{aligned} 
\end{equation}
is generally positive, because $\alpha,\varphi\in[0,\pi]$
and the bivector part 
\begin{equation}
\begin{aligned}
\langle e^{\frac\beta2\vec R}\rangle_2
=&\sin(\frac\alpha2)\cos(\frac\varphi2)\vec P-\frac{1}{||\langle e^{\varphi\vec Q}\rangle_2||}
\cos(\frac\alpha2)\sin(\frac\varphi2)\sin(\alpha)||\v_{\parallel\vec P}||_{L^2}^2\vec P
\\&-\frac{2}{||\langle e^{\varphi\vec Q}\rangle_2||}\sin(\frac\alpha2)\sin(\frac\varphi2)\vec P\int\v_{\parallel\vec P}\v_{\perp\vec P}\d \x
\end{aligned} 
\end{equation}
has the squared norm
\begin{equation}
\begin{aligned}\label{bivnorm}
||\langle e^{\frac\beta2\vec R}\rangle_2||^2
=&\sin(\frac\alpha2)^2\cos(\frac\varphi2)^2
-\frac{1}{||\langle e^{\varphi\vec Q}\rangle_2||}\cos(\frac\varphi2)\sin(\frac\varphi2)\sin(\alpha)^2||\v_{\parallel\vec P}||_{L^2}^2
\\&+\frac{1}{||\langle e^{\varphi\vec Q}\rangle_2||^2}\cos(\frac\alpha2)^2\sin(\frac\varphi2)^2\sin(\alpha)^2||\v_{\parallel\vec P}||_{L^2}^4
\\&+\frac{4}{||\langle e^{\varphi\vec Q}\rangle_2||^2}\sin(\frac\alpha2)^2\sin(\frac\varphi2)^2||\int\v_{\parallel\vec P}\v_{\perp\vec P}\d \x||^2.
\end{aligned} 
\end{equation}
For the next inequalities we use, that all appearing parts are positive and that the tangent and the quadratic function are monotonically increasing for positive arguments. We get
\begin{equation}
\begin{aligned}
\frac\beta2\leq\frac\alpha2
&\Leftrightarrow
\arctan(\frac{||\langle e^{\frac\beta2\vec R}\rangle_2||}{\langle e^{\frac\beta2\vec R}\rangle_0})
\leq\arctan(\frac{\sin(\frac\alpha2)}{\cos(\frac\alpha2)})
\\&\Leftrightarrow
\frac{||\langle e^{\frac\beta2\vec R}\rangle_2||}{\langle e^{\frac\beta2\vec R}\rangle_0}
\leq\frac{\sin(\frac\alpha2)}{\cos(\frac\alpha2)}
\\&\Leftrightarrow
||\langle e^{\frac\beta2\vec R}\rangle_2||^2\cos(\frac\alpha2)^2||\langle e^{\varphi\vec Q}\rangle_2||^2
\leq\sin(\frac\alpha2)^2\langle e^{\frac\beta2\vec R}\rangle_0^2||\langle e^{\varphi\vec Q}\rangle_2||^2.
\end{aligned} 
\end{equation}
Now we insert the scalar part (\ref{sc}) and the bivector norm (\ref{bivnorm})
\begin{equation}
\begin{aligned}
&\sin(\frac\alpha2)^2\cos(\frac\alpha2)^2\cos(\frac\varphi2)^2||\langle e^{\varphi\vec Q}\rangle_2||^2
+\sin(\frac\varphi2)^2\cos(\frac\alpha2)^4\sin(\alpha)^2||\v_{\parallel\vec P}||_{L^2}^4
\\&-||\langle e^{\varphi\vec Q}\rangle_2|| \cos(\frac\varphi2)\sin(\frac\varphi2)\cos(\frac\alpha2)^2\sin(\alpha)^2||\v_{\parallel\vec P}||_{L^2}^2
\\&+4\sin(\frac\varphi2)^2\sin(\frac\alpha2)^2\cos(\frac\alpha2)^2||\int\v_{\parallel\vec P}\v_{\perp\vec P}\d \x||^2
\\\leq
&\sin(\frac\alpha2)^2\cos(\frac\alpha2)^2\cos(\frac\varphi2)^2||\langle e^{\varphi\vec Q}\rangle_2||^2
+\sin(\frac\alpha2)^4\sin(\frac\varphi2)^2\sin(\alpha)^2||\v_{\parallel\vec P}||_{L^2}^4
\\&+2\cos(\frac\alpha2)\cos(\frac\varphi2) \sin(\frac\alpha2)^3\sin(\frac\varphi2)\sin(\alpha)||\langle e^{\varphi\vec Q}\rangle_2||\,||\v_{\parallel\vec P}||_{L^2}^2,
\end{aligned} 
\end{equation}
remove the identical parts, identify $2\sin(\frac\alpha2)\cos(\frac\alpha2)=\sin(\alpha)$, divide both sides by $\sin(\frac\varphi2)\sin(\alpha)^2$, and get
\begin{equation}
\begin{aligned}\label{bew_b>a_absch}
\Leftrightarrow& \sin(\frac\varphi2)\cos(\frac\alpha2)^4||\v_{\parallel\vec P}||_{L^2}^4
-||\langle e^{\varphi\vec Q}\rangle_2|| \cos(\frac\varphi2)\cos(\frac\alpha2)^2||\v_{\parallel\vec P}||_{L^2}^2
\\&+\sin(\frac\varphi2)||\int\v_{\parallel\vec P}\v_{\perp\vec P}\d \x||^2
\\&\leq
\sin(\frac\alpha2)^4\sin(\frac\varphi2)||\v_{\parallel\vec P}||_{L^2}^4
+\cos(\frac\varphi2) \sin(\frac\alpha2)^2||\langle e^{\varphi\vec Q}\rangle_2||\,||\v_{\parallel\vec P}||_{L^2}^2.
 \\\Leftrightarrow&
 \sin(\frac\varphi2)\big((\cos(\frac\alpha2)^4-\sin(\frac\alpha2)^4)||\v_{\parallel\vec P}||_{L^2}^4
+||\int\v_{\parallel\vec P}\v_{\perp\vec P}\d \x||^2\big)
\\&\leq
\cos(\frac\varphi2) ||\langle e^{\varphi\vec Q}\rangle_2||( \sin(\frac\alpha2)^2+\cos(\frac\alpha2)^2)||\v_{\parallel\vec P}||_{L^2}^2
\\\Leftrightarrow&
 \sin(\frac\varphi2)(\cos(\alpha)||\v_{\parallel\vec P}||_{L^2}^4
+||\int\v_{\parallel\vec P}\v_{\perp\vec P}\d \x||^2)
\leq
\cos(\frac\varphi2) ||\langle e^{\varphi\vec Q}\rangle_2||\,||\v_{\parallel\vec P}||_{L^2}^2.
\end{aligned} 
\end{equation}
In (\ref{bew_b>a_absch}) we have identified $\cos(\frac\alpha2)^4-\sin(\frac\alpha2)^4=\cos(\alpha)$ and $\sin(\frac\alpha2)^2+\cos(\frac\alpha2)^2=1$. We further use (\ref{scphiq}) to replace $\cos(\alpha)$ by $(\cos(\varphi)-||\v_{\perp\vec P}||_{L^2}^2)||\v_{\parallel\vec P}||_{L^2}^{-2}$ and (\ref{bivphiq}) to replace $||\langle e^{\varphi\vec Q}\rangle_2||$ by $\sin(\varphi)$, which leads to
\begin{equation}
\begin{aligned}\label{beweisende}
\Leftrightarrow
&\sin(\frac\varphi2)\big((\cos(\varphi)-||\v_{\perp\vec P}||_{L^2}^2)||\v_{\parallel\vec P}||_{L^2}^2
+||\int\v_{\parallel\vec P}\v_{\perp\vec P}\d \x||^2\big)
\\&\leq
\cos(\frac\varphi2) \sin(\varphi)||\v_{\parallel\vec P}||_{L^2}^2
\\\Leftrightarrow
&\sin(\frac\varphi2)\big((2\cos(\frac\varphi2)^2-1-||\v_{\perp\vec P}||_{L^2}^2)||\v_{\parallel\vec P}||_{L^2}^2
+||\int\v_{\parallel\vec P}\v_{\perp\vec P}\d \x||^2\big)
\\&\leq
2\sin(\frac\varphi2)\cos(\frac\varphi2)^2||\v_{\parallel\vec P}||_{L^2}^2
\\\Leftrightarrow
&\sin(\frac\varphi2)\big(-||\v_{\parallel\vec P}||_{L^2}^2-||\v_{\parallel\vec P}||_{L^2}^2||\v_{\perp\vec P}||_{L^2}^2
+||\int\v_{\parallel\vec P}\v_{\perp\vec P}\d \x||^2\big)
\leq
0.
\end{aligned} 
\end{equation}
The Cauchy Schwartz inequality (CSI) of $L^2(\R)$ guarantees
\begin{equation}
\begin{aligned}\label{csi}
||\int\v_{\parallel\vec P}(\x)\v_{\perp\vec P}(\x)\d\x||^2
\leq&
(\int ||\v_{\parallel\vec P}(\x)\v_{\perp\vec P}(\x)||\d\x)^2
\\=&
(\int ||\v_{\parallel\vec P}(\x)||\,||\v_{\perp\vec P}(\x)||\d\x)^2
\\=&\langle||\v_{\parallel\vec P}(\x)||,||\v_{\perp\vec P}(\x)||\rangle_{L^2}^2
\\\overset{\text{CSI}}\leq&
||\v_{\parallel\vec P}||_{L^2}^2||\v_{\perp\vec P}||_{L^2}^2,
\end{aligned} 
\end{equation}
so we know, that the part in the brackets in the last line of (\ref{beweisende}) is always negative. The sine is always positive for $\varphi\in[0,\pi]$. Therefore in the shape of (\ref{beweisende}) it is easy to recognize that the inequality $\beta\leq\alpha$ is generally fulfilled.\smartqed\qed
\end{proof}
We want to apply Lemma \ref{l:outer_prod} repeatedly and construct a series of decreasing angles, describing the remaining misalignment of the vector fields. The next theorem shows that the misalignment vanishes by iteration.
\begin{thm}\label{t:conv2}
For a square integrable vector field $\v\in L^2(\R^m,\R^{3,0}\subset\clifford{3,0})$ let $\beta:[0,\pi)\to[0,\pi)$ be a function defined by $\beta(\alpha)=2\arg(e^{\frac\alpha2 \vec P}e^{\frac\varphi2 \vec Q})$ with 
\begin{equation}
\begin{aligned}
e^{\varphi\vec Q}=\frac{(\operatorname{R} _{P,\alpha}(\v)\star\v)(0)}{{||\operatorname{R} _{P,\alpha}(\v)||_{L^2}||\v||_{L^2}}}.
\end{aligned} 
\end{equation}
Then the series $\alpha_0=\alpha,\alpha_{n+1}=\beta(\alpha_{n})$ converges to zero for all $\alpha\in[0,\pi)$.
\end{thm}
\begin{proof}
If $\alpha=0$ or $||\v_{\parallel\vec P}||_{L^2}=0$ the series is trivial because $\operatorname{R} _{P,\alpha}(\v)=\v$ almost everywhere. From now on let $\alpha\neq0$ and $||\v_{\parallel\vec P}||_{L^2}\neq0$. Lemma \ref{l:outer_prod} shows that the magnitudes of the series are monotonically decreasing. Since they are bound from below by zero the series is convergent. We denote the limit by $a=\lim_{n\to\infty}\alpha_n$. The function $\beta(\alpha)$ is continuous, so we can swap it and the limit
\begin{equation}
\begin{aligned}\label{beta=alpha}
a=\lim_{n\to\infty}\alpha_{n+1}=\lim_{n\to\infty}\beta(\alpha_n)=\beta(\lim_{n\to\infty} \alpha_n)=\beta(a).
\end{aligned} 
\end{equation}
This equality is the sharp case of the inequality (\ref{beta<alpha}). For $\alpha\neq0$ also $\varphi\neq0$, because from (\ref{scphiq}), $||\v_{\parallel\vec P}||_{L^2}^2+||\v_{\perp\vec P}||_{L^2}^2=1$, and $||\v_{\parallel\vec P}||_{L^2}\neq0$ we get
\begin{equation}
\begin{aligned}\label{alpha0phi0}
\varphi=0
\Leftrightarrow
1=\cos(\varphi)=\cos(\alpha)||\v_{\parallel\vec P}||_{L^2}^2+||\v_{\perp\vec P}||_{L^2}^2
\Leftrightarrow
\cos(\alpha)=1
\Leftrightarrow
\alpha=0
\end{aligned} 
\end{equation} 
Therefore the transformative steps that lead from (\ref{beta<alpha}) to (\ref{beweisende}) in the proof of Lemma \ref{l:outer_prod} are biconditional. So analogously to these steps (\ref{beta=alpha}) is equivalent to
\begin{equation}
\begin{aligned}
&\sin(\frac\varphi2)(-||\v_{\parallel\vec P}||_{L^2}^2-||\v_{\parallel\vec P}||_{L^2}^2||\v_{\perp\vec P}||_{L^2}^2
+||\int\v_{\parallel\vec P}\v_{\perp\vec P}\d \x||^2)
=
0
\end{aligned} 
\end{equation}
The part in the brackets is strictly negative, because of the Cauchy Schwartz inequality (\ref{csi}) and $||\v_{\parallel\vec P}||_{L^2}\neq0$, therefore equality can only occur for $\sin(\frac\varphi2)=0$, which means $\varphi=0$. Like in (\ref{alpha0phi0}) this leaves $a=0$ as the only possible limit.\smartqed\qed
\end{proof}
\section{Algorithm and Experiments}
Motivated by Theorem \ref{t:conv2} we present Algorithm \ref{alg3} for the iterative detection of outer rotations of vector fields using geometric cross correlation. It has been designed with attention to the efficient use of memory and to handle possible exceptions. In the case of $\alpha=\pi$ the correlation might be real valued and can not be distinguished from the cases where no rotation is necessary. To fix this exception we suggest an artificial disturbance after the first step of the algorithm, compare Line 7 in Algorithm \ref{alg3}. If it is not real the new misalignment will be smaller, because any misalignment is smaller than $\pi$. 
In all other cases Theorem \ref{t:conv2} guarantees the convergence. Our results also apply to the geometric product of the vector fields at any position $\x\in\R^m,\v_{\parallel\vec P}(\x)\neq0$, but we prefer the geometric correlation because of its robustness.
\begin{cor}
Algorithm \ref{alg3} returns the correct rotational misalignment for any three-dimensional linear vector field and its copy generated from an arbitrary outer rotation.
\end{cor}
%
\begin{algorithm}
\caption{Detection of outer misalignment of vector fields in 3D}
\label{alg3}
\begin{algorithmic}[1]
\REQUIRE vector field: $\v(\x)$, rotated pattern: $\u(\x)$, desired accuracy: $\varepsilon>0$,
\STATE $\varphi=\pi,\alpha=0, \vec P=\e_{12},iter=0$,
\WHILE{$\varphi>\varepsilon$}
  \STATE $iter++$,
  \STATE $Cor=(\u(\x)\star \v(\x))(0)$,
  \STATE $\varphi=\arg(Cor)$,
  \STATE $\vec Q=\langle Cor\rangle_2|\langle Cor\rangle_2|^{-1}$,
  \IF {$iter=1$ and $\varphi=0$}
    \STATE $\varphi=\pi/4$,
    \STATE $\vec Q=\e_{12}$,
  \ENDIF
  \STATE $\u(\x)=e^{-\frac\varphi2\vec Q}\u(\x)e^{\frac\varphi2\vec Q}$,
  \STATE $\alpha^\prime=2\arg(e^{\frac\alpha2\vec P}e^{\frac\varphi2\vec Q})$,
  \STATE $\vec P=\langle e^{\frac\alpha2\vec P}e^{\frac\varphi2\vec Q}\rangle_2|\langle e^{\frac\alpha2\vec P}e^{\frac\varphi2\vec Q}\rangle_2|^{-1}$,
  \STATE $\alpha=\alpha^\prime$,
 \ENDWHILE
\ENSURE angle: $\alpha$, plane: $\vec P$, corrected pattern: $\u(\x)$, iterations needed: $iter$.
\end{algorithmic}
\end{algorithm}
%
We practically tested Algorithm \ref{alg3} applying it to continuous, linear vector fields $\R^3\to\clifford{3,0}$, that vanish outside the unit square. The vector fields were determined by nine random coefficients with magnitude not bigger than one. The plane and the angle $\alpha\in[0,\pi]$ of the outer rotation were also chosen randomly. The results for one million applications can be found in Table \ref{tab:1}. The error was measured from the square root of the sum of the squared differences of the determined and the given coefficients. The experiments showed that 
high numbers of necessary iterations are much more likely to happen for angles with high magnitude and that the average error decreases linearily with the demanded accuracy. But most importantly we observed that Algorithm \ref{alg3} converged in all linear cases, just as the theory suggested. The hypercomplex correlation method of Moxey et. al. in \cite{MSE03} can be interpreted as one step of Algorithm \ref{alg3}, i.e., to terminate without iteration. The last row of Table \ref{tab:1} shows how the application of more iterations increases the accuracy.
%

\begin{table}
\caption{Results of Algorithm \ref{alg3} depending on the required accuracy}
\label{tab:1}       
%
%
\begin{tabular}{p{4.05cm}p{1.4cm}p{1.4cm}p{1.4cm}p{1.4cm}p{1.4cm}}
\hline\noalign{\smallskip}
determined accuracy $eps$	&0.1	& 0.01	& 0.001	& 0.0001	& 0.00001\\
\noalign{\smallskip}\svhline\noalign{\smallskip}
average error			&0.17	& 0.02	& 0.002	& 0.0002	& 0.00002\\
maximal error			&2.49	& 1.29	& 0.47	& 0.12		& 0.001\\
average number of iterations	&4.23	& 11.76	& 21.44	& 31.78		& 42.26\\
\noalign{\smallskip}\hline\noalign{\smallskip}
\end{tabular}
\end{table}
\section{Conclusions and Outlook}
The geometric cross correlation of two vector fields is scalar and bivector valued. Moxey et. al. \cite{MSE03} realized that this rotor yields an approximation of the outer rotational misalignment of vector fields. We analyzed that the quality of this approximation depends on the parallel and the orthogonal parts of the fields and proved in Lemma \ref{l:outer_prod} that the application of this rotor to the outer rotated copy of any vector field never increases the misalignment to the original field. In Theorem \ref{t:conv2} we refined this fact and showed that iterative application completely erases the misalignment of the rotationally misaligned vector fields. We presented Algorithm \ref{alg3}, which additionally contains exception handling, and experimentally confirmed our theoretical findings for linear vector fields. 
\par
We currently analyze the application of this approach to total rotations in \cite{BSH12b}.
Other ideas to pursue are extensions of our results to higher dimensions or arbitrary multivector fields. Practically it will be of major interest how the algorithm is able to deal with vector fields, that are discrete, disturbed, differ by global translations, or are non-linear and if it is able to minimize the squared differences if the fields are not equal after rotation but only similar. In our future work we want to use the algorithm for image registration of real world data, accelerate it by means of a fast Fourier transform and a geometric convolution theorem \cite{BSH12c}, and compare it to established methods like, invariant descriptors, phase correlation, the Kabsch algorithm, methods using spherical harmonics, logpolar coordinates, or wavelets with standard and geometric algebra interpretations, with respect to reliability and runtime. 
  \bibliographystyle{plain} 
  \bibliography{./../../Literaturverzeichnis}

\end{document}